\documentclass{amsart}
\usepackage{amsfonts,amsmath,amsthm,amssymb}
\usepackage[all]{xy}
\usepackage[pagebackref, colorlinks=true, linkcolor=blue, citecolor=blue]{hyperref}

\numberwithin{equation}{section}

\begin{document}

\newtheorem{theorem}{Theorem}[section]
\newtheorem{lemma}[theorem]{Lemma}
\newtheorem{proposition}[theorem]{Proposition}
\newtheorem{corollary}[theorem]{Corollary}

\theoremstyle{definition}
\newtheorem{definition}[theorem]{Definition}
\newtheorem{example}[theorem]{Example}

\theoremstyle{remark}
\newtheorem{remark}[theorem]{Remark}

\newcommand{\stirl}[2]{\genfrac{\{}{\}}{0pt}{}{#1}{#2}}
\newcommand{\dstirl}[2]{\genfrac{\{}{\}}{0pt}{0}{#1}{#2}}

\renewcommand{\bar}{\overline}
\newcommand{\de}{\partial}
\newcommand{\per}{\!\cdot\!}

\newcommand{\Aff}{\operatorname{Aff}}
\newcommand{\Hom}{\operatorname{Hom}}
\newcommand{\End}{{\operatorname{End}}}
\newcommand{\Diff}{\operatorname{Diff}}
\newcommand{\Coder}{\operatorname{Coder}}
\newcommand{\Q}{\mathbb{Q}}
\newcommand{\Z}{\mathbb{Z}}
\newcommand{\K}{\mathbb{K}\,}

\title{Uniqueness and intrinsic properties of non-commutative Koszul brackets}
\author{Marco Manetti}

\address{\newline
Universit\`a degli studi di Roma La Sapienza,\hfill\newline
Dipartimento di Matematica \lq\lq Guido
Castelnuovo\rq\rq,\hfill\newline
P.le Aldo Moro 5,
I-00185 Roma, Italy.}
\email{manetti@mat.uniroma1.it}
\urladdr{www.mat.uniroma1.it/people/manetti/}

\renewcommand{\subjclassname}{%
\textup{2010} Mathematics Subject Classification}


\date{May 6, 2016}

\begin{abstract}
There exists a unique  natural extension of  higher Koszul brackets to every unitary associative algebra in a way  that every square zero operator of degree 1 gives a curved $L_{\infty}$ structure. 
\end{abstract}
\maketitle

\section*{Introduction}

The name (higher, commutative) Koszul brackets is usually referred to the sequence of graded symmetric maps
$\Psi^n_f\colon A^{\odot n}\to A$, $n\ge 0$, defined in \cite{koszul85} for every graded commutative unitary algebra $A$ and every linear endomorphism $f\colon A\to A$ by the formula:
\begin{equation}\label{equ.defiKoszul} 
\begin{split}
\Psi^0_f=\;&f(1),\\
\Psi^1_f(a)=\;&f(a)-f(1)a,\\
\Psi^2_f(a,b)=\;&
f(ab)-f(a)b-(-1)^{|a||b|}f(b)a+f(1)ab\\
\vdots\;\;&\\
\Psi^n_f(a_1,\ldots,a_n)=\;&
\sum_{k=0}^n\frac{(-1)^{n-k}}{k!(n-k)!}\sum_{\pi\in \Sigma_n}\epsilon(\pi)\; 
f(1\cdot a_{\pi(1)}\cdots a_{\pi(k)})a_{\pi(k+1)}\cdots a_{\pi(n)}\,,
\end{split}\end{equation}
where $\epsilon(\pi)$ is the Koszul sign of the permutation $\pi$ with respect to the sequence of homogeneous elements
$a_1,\ldots,a_n$. As proved in \cite{BDA,kravchenko2000,voronov} they have the remarkable property of 
satisfying the generalized Jacobi identities  of Lada and Stasheff \cite{LadaStas}, and therefore 
they are applied in the study 
of $L_{\infty}$-algebras, of (commutative) Batalin-Vilkovisky algebras and their deformations.

The question of extending their definition to every unitary graded associative algebra, preserving  generalized Jacobi identities, is a nontrivial task and has been first answered by Bering \cite{Bering} about ten years ago. Very recently, other solutions, quite different in their origin and presentation, are proposed  by Bandiera \cite{derived,kapranovbrackets} 
and by Manetti and Ricciardi \cite{GM}.

Apart from the natural question whether the above mentioned non-commutative extension formulas coincide or not, the main goal of this paper is to determine 
a minimal set of conditions which implies existence and unicity of 
non-commutative Koszul brackets.  

A similar goal has been recently achieved by Markl in the paper \cite{markl1}, 
where it is proved that both  hierarchies of  B\"{orjeson} brackets and 
commutative Koszul brackets are the unique natural hierarchies 
of brackets satisfying the technical conditions called hereditarity, recursivity and 
with fixed initial terms. 
Strictly speaking, this viewpoint does not apply to our goal since it is easy to see that a  
non-commutative hereditary extension of Koszul brackets cannot satisfy generalized Jacobi identities; however Markl's work has certainly inspired this paper. 

The point of view which we adopt in this paper is based on the slogan that 
``the most important properties for a hierarchy of brackets are naturality, base change and generalized  Jacobi formulas''. Precise definitions  will be given in Theorem~\ref{thm.unicity}; here we only mention that Markl's notion of naturality is essentially equivalent to the join of our notions of naturality and base change.  

Whereas in the previous literature on the subject the starting point is Koszul's definition in the commutative case, the middle point is the proposal of a non-commutative extension and the conclusive point is the proof of generalized Jacobi identities, in this paper 
we reverse the logical path: we start with a generic hierarchy 
satisfying naturality, base change and generalized Jacobi, and then we add assumptions on the  initial terms until we reach the  unicity. Quite surprisingly, this  
approach goes very smoothly and provides, in our opinion, a  simplification of the theory also 
when restricted to the classical commutative case.     
  
The paper is organized as follows: in Section~\ref{sec.setup} we fix notation and we recall the definition of the Nijenhuis-Richardson bracket, in terms of which the generalized Jacobi identities can be expressed in their simplest form. 
Section~\ref{sec.uniquenesstheorem} is completely devoted to the proof of the uniqueness theorem of Koszul brackets, whose first properties, including their restriction to the commutative case, are studied in Section~\ref{sec.firstproperties}. In Section~\ref{sec.beringbandiera} we shall prove that the 
non-commutative Koszul brackets may be also explicitly described by  the formulas given in \cite{derived,Bering}. The last section is devoted to a discussion about 
the reduction of  Koszul brackets to non-unitary  graded associative algebras.

\subsection*{Acknowledgments} 
The author thanks the referee for several useful comments and acknowledges partial support by Italian MIUR under PRIN project 2012KNL88Y ``Spazi di moduli e teoria di Lie''.

\bigskip
\section{General setup}
\label{sec.setup}

The symmetric group of permutations of $n$ elements is denoted by $\Sigma_n$. Every graded vector space, every graded algebra and every tensor product is intended $\Z$-graded and
over a fixed field $\K$ of characteristic 0. For every graded vector space $V$
we shall denote by $V^{\odot n}$ its $n$th symmetric power: for simplicity of notation we always identify a linear map 
$f\colon V^{\odot n}\to W$ with the corresponding graded symmetric operator 
\[f\colon \underbrace{V\times\cdots\times V}_{n\text{ factors}}\to W,\qquad f(v_1,\ldots,v_n)=f(v_1\odot\cdots\odot v_n).\]
 
To every graded vector space $V$ we shall consider the following  graded Lie  algebras: 
\begin{enumerate}

\item the algebra of linear endomorphisms: 
\[\End^*(V)=\Hom^*_{\K}(V,V)=\bigoplus_{n\in\Z}\Hom^n_{\K}(V,V)\,,\]
equipped with the graded commutator bracket;

\item the space  of  affine  endomorphisms:
\[\Aff^*(V)=\bigoplus_{n\in\Z}\Aff^n(V)=\{f\in \End^*(V\oplus \K)\mid f(V\oplus \K)\subseteq V\}\,,\]
considered as a graded  Lie subalgebra of $\End^*(V\oplus \K)$.
\end{enumerate} 
Thus, giving an element  $f\in \Aff^n(V)=\{f\in \End^n(V\oplus \K)\mid f(V\oplus \K)\subseteq V\}$ is the same as giving a linear map $g_i\colon V^i\to V^{i+n}$, $g_i(v)=f(v)$, for every $i\not=0$ and an affine map $g_0\colon V^0\to V^n$, $g_0(v)=f(v+1)$.

It is useful  to consider both $\End^*(V)$ and $\Aff^*(V)$ as graded Lie subalgebras of 
\[D(V)=\prod_{n\ge -1}D_n(V),\qquad D_{-1}(V)=V,\quad D_n(V)=\Hom^*_{\K}(V^{\odot n+1},V),\]
where the Lie structure on $D(V)$ is given by  the Nijenhuis-Richardson bracket,  
induced by the right pre-Lie product 
$\barwedge$
defined in the following way \cite{NijRich67}:
given $f\in D_n(V)$ and $g\in D_m(V)$ the operator  
\[ f\barwedge g\in D_{n+m}(V)\]
is equal to:
\begin{enumerate} 

\item $f\barwedge g=0$ whenever $f\in D_{-1}(V)=V$;

\item $f\barwedge g
(v_{1},\ldots, v_{n})=f(g,v_{1},\ldots, v_{n})$ whenever $g\in D_{-1}(V)=V$;

\item when $n,m\ge 0$ we have
\[ f\barwedge g
(v_{0},\ldots, v_{n+m})=\!\!\!
\sum_{\sigma\in S(m+1,n)}\epsilon(\sigma)
f(g(v_{\sigma(0)},\ldots, v_{\sigma(m)}), v_{\sigma(m+1)},\ldots, 
v_{\sigma(m+n)})\,.\]
\end{enumerate}

Here $S(m+1,n)\subset\Sigma_{n+m+1}$ is the set of shuffles of type $(m+1,n)$, i.e., the set of permutations $\sigma$ of $0,\ldots,n+m$ 
such that $\sigma(0)<\cdots<\sigma(m)$ and $\sigma(m+1)<\cdots<\sigma(m+n)$. The Koszul sign 
$\epsilon(\sigma)$ is equal to $(-1)^{\alpha}$, where $\alpha$ is the number of pairs $(i,j)$ such that 
$i<j$, $\sigma(i)>\sigma(j)$ and $|v_i||v_j|$ is odd.
The Nijenhuis-Richardson bracket is defined as the graded commutator of $\barwedge$:
\[ [f,g]=f\barwedge g-(-1)^{|f||g|}g\barwedge f\,. \]

Notice that, since $[D_i(V),D_j(V)]\subseteq D_{i+j}(V)$ we have 
that $D_0(V),D_{-1}(V)\times D_0(V)$ and $D_{\ge 0}(V)=\prod_{n\ge 0}D_n(V)$ are graded Lie subalgebras of $D(V)$; notice also that $[f,Id_V]=n f$ for every $f\in D_n(V)$, where $Id_V$ is the identity on $V$.

By definition $\End^*(V)=D_0(V)$ and there exists a natural isomorphism of graded Lie algebras  
$\Aff^*(V)\cong D_{-1}(V)\times D_0(V)$, where  every pair $(x,f)\in D_{-1}(V)\times D_0(V)$ corresponds to the linear map 
\[ (x,f)\colon V\oplus \K\to V, \qquad (x,f)(v+t)=f(v)+tx,\qquad v\in V,\; t\in \K\;.\]

\begin{remark}\label{rem.NRvsGerstenhaber}
It is well known, and in any case easy to prove, that the  Nijenhuis-Richardson product $\barwedge$ is the symmetrization of the Gerstenhaber product
\[ \Hom^*_{\K}(V^{\otimes p-n+1},V)\times \Hom^*_{\K}(V^{\otimes n+1},V)\xrightarrow{\;\circ\;}
\Hom^*_{\K}(V^{\otimes p+1},V),\]
\[ f\circ g(v_0,\ldots,v_{p})=\sum_{i=0}^{p-n}(-1)^{|g|(|v_0|+\cdots+|v_{i-1}|)} 
f(v_0,\ldots,v_{i-1},g(v_{i},\ldots,v_{i+n}),v_{i+n+1},\ldots,v_{p}).\]
More precisely, denoting by $N\colon V^{\odot n+1}\to V^{\otimes n+1}$ the  map 
\[ N(v_0\odot\cdots\odot v_n)=\sum_{\sigma\in \Sigma_{n+1}}\epsilon(\sigma) 
v_{\sigma(0)}\otimes \cdots\otimes v_{\sigma(n)}\,,\]
we have
$(f\circ g)N=fN\barwedge gN$.
\end{remark}

\begin{remark} Although not relevant for this paper, it is useful to point out that 
the graded Lie algebra $D(V)$ is naturally isomorphic to the graded Lie algebra of coderivations of the symmetric coalgebra $S^c(V)=\bigoplus_{n\ge 0}V^{\odot n}$. The isomorphism 
\[ \Coder^*(S^c(V),S^c(V))\to D(V)\cong\Hom^*_{\K}(S^c(V),V)\]
is induced by taking composition  with the projection map $S^c(V)\to V$, see e.g., 
\cite{K,LadaStas,sta93}.
\end{remark}

\begin{definition}
For  a unitary graded associative algebra $A$ we 
consider the sequence of maps 
$\mu_n\in D_n(A)$, $n\ge -1$:
\[ \mu_{-1}=1,\quad \mu_0=Id_A,\qquad \mu_n(a_0,\ldots,a_n)=\frac{1}{(n+1)!}\sum_{\sigma\in \Sigma_{n+1}}\epsilon(\sigma) 
a_{\sigma(0)}a_{\sigma(1)}\cdots a_{\sigma(n)}\,.\] 
\end{definition}

When $A$ is graded commutative we recover the multiplication maps $\mu_n(a_0,\ldots,a_n)=a_0\cdots a_n$.
When the algebra $A$ is clear from the context, we shall simply denote by $Id$ the identity map 
$Id_A\colon A\to A$.   In order to avoid possible confusion with the Nijenhuis-Richardson bracket we shall denote the graded commutator of $a,b\in A$ by $\{a,b\}=ab-(-1)^{|a||b|}ba\in A$  .

The following lemma is a straightforward consequence of Remark~\ref{rem.NRvsGerstenhaber}.

\begin{lemma}\label{lem.commutazionemuenne} 
In the above setup,  for every $n,m\ge -1$ we have  
\[\mu_n\barwedge \mu_m=\dbinom{n+m+1}{m+1}\mu_{n+m},\qquad
[\mu_n,\mu_m]=(n-m)\frac{(n+m+1)!}{(n+1)!(m+1)!}\mu_{n+m}\;.\]
\end{lemma}

\bigskip
\section{The uniqueness theorem}
\label{sec.uniquenesstheorem}

\begin{theorem}\label{thm.unicity}  
There exists a unique way to assign to every  unitary graded associative algebra $A$  a morphism of graded vector spaces 
\[ \Psi\colon\Aff^*(A)\to D(A),\qquad x\mapsto \Psi_x=\sum_{n=0}^{\infty} \Psi^n_x,\quad \Psi^n_x\in D_{n-1}(A),\] 
such that the following conditions are satisfied:
\begin{enumerate}

\item \emph{generalized Jacobi:} $\Psi$ is a morphism of graded Lie algebras;

\item \emph{naturality:} for every morphism   
$\alpha\colon A\to B$ of unitary graded algebras, for every $x\in A$ and every pair of linear maps $f\colon A\to A$, $g\colon B\to B$ such that $g\alpha=\alpha f$, we have
\[ \alpha\Psi^n_x=\Psi^n_{\alpha(x)}\alpha^{\odot n},\quad
\alpha\Psi^n_f=\Psi^n_g\alpha^{\odot n}\quad\colon A^{\odot n}\to B\;.\]

\item \emph{base change}:  the operators $\Psi^n_1,\Psi^n_{Id}$ are multilinear over the centre of $A$. More precisely, if $c\in A$ is homogeneous and $ac=(-1)^{|a||c|}ca$ for every homogeneous $a\in A$, then 
\[ \qquad\Psi^n_1(a_1,\ldots,a_nc)=\Psi^n_1(a_1,\ldots,a_n)c,\qquad 
\Psi^n_{Id}(a_1,\ldots,a_nc)=\Psi^n_{Id}(a_1,\ldots,a_n)c,
\]
for every $a_1,\ldots,a_n$.

\item \emph{initial terms:} for every $x\in A$, $f\in \End^*(A)$, we have
\[\Psi^0_x=x,\qquad \Psi^0_f=f(1)\;.\]

\item \emph{gauge fixing:} for $A=\K$ we have $\Psi^n_{Id}=0$ for every $n>0$.
\end{enumerate}
\end{theorem}

\begin{proof} We   identify  
$\Aff^*(A)$ with the graded Lie subalgebra $D_{-1}(A)\times D_0(A)\subset D(A)$;
for our goals it is convenient to prove the existence following the ideas  of 
\cite{kapranovbrackets,GM}. 
Notice first that every operator $\mu_n$ is multilinear over the centre of $A$ and commutes 
with  morphisms of unitary 
graded associative algebras.
Next, for every sequence $K_1,K_2,\ldots$ of rational numbers, the map 
\[ \widehat{\Psi}\colon D(A)\to D(A),\qquad 
\widehat{\Psi}_u=\exp\left(\left[-,\sum_{n=1}^{\infty}K_n \mu_n
\right]\right)\exp([-,\mu_{-1}])u,\]  
is an  isomorphism of graded Lie algebras which is compatible with morphisms of unitary graded algebras and gives the required initial terms. 
A simple recursive argument shows that the gauge fixing condition 
\[A=\K,\qquad \widehat{\Psi}_{\mu_0}=\exp\left(\left[-,\sum_{n=1}^{\infty}K_n \mu_n
\right]\right)(\mu_0+\mu_{-1})=\mu_{-1},\]
can be written as 
\[\exp\left(\left[\sum_{n=1}^{\infty}K_n \mu_n, -
\right]\right)\mu_{-1}=\mu_{-1}+\mu_0,\]
and determines uniquely the coefficients $K_n$. The first terms are: 
\[K_1=1,\quad K_2=-\frac{1}{2},\quad  K_3=\frac{1}{2},\quad K_4=\-\frac{2}{3},\quad 
K_5=\frac{11}{12},\quad K_6=-\frac{3}{4},\quad  K_7=-\frac{11}{6},\;\ldots\,.\]
According to \cite{GM}, the formal power series 
\[ \sum_{n\ge 1}K_n\frac{t^{n+1}}{(n+1)!}\in \Q[[t]]\]
is the iterative logarithm of $e^t-1$, cf. \cite{AB}, and the sequence $K_n$ 
may be also computed recursively by the linear equations
\[ K_1=1,\qquad K_n=\frac{-2}{(n+2)(n-1)}\sum_{i=1}^{n-1} \stirl{n+1}{i}K_i\;,\]
where $\stirl{n+1}{i}$ are the Stirling numbers of the second kind.
It is now sufficient to define $\Psi$ as the restriction  $\widehat{\Psi}$ to the  graded Lie subalgebra 
$D_{-1}(A)\times D_0(A)$. 

\medskip

Let us now prove the unicity, the first step is to prove, for every algebra $A$, the formulas: 
\begin{equation}\label{equ.formulapsiuno}
\Psi_1=\sum_{n\ge 0}(-1)^n\mu_{n-1},\qquad \Psi_{Id}=\mu_{-1}\,.
\end{equation}
Assume first $A=\K$, then $\mu_n$ is a generator of $D_n(\K)$  and therefore there exists a sequence $s_0,s_1,\ldots$ in $\K$ such that
\[ \Psi_1=\sum_{n\ge 0}s_n\mu_{n-1}\,,\] 
where $s_0=1$ by the initial terms condition. Using the relation
$[\Psi_{Id},\Psi_1]=\Psi_{[Id,1]}=\Psi_1$ we obtain $\Psi^n_1=[\mu_{-1},\Psi^{n+1}_1]$ for every $n\ge -1$ and then

\[ s_{n}\mu_{n-1}=[\mu_{-1},s_{n+1}\mu_n]=- s_{n+1}\mu_{n-1},\qquad 
s_{n+1}=-s_n=(-1)^{n+1}\;.\]
Consider now the polynomial algebra $\K[t]$, with $t$ a central element of degree $0$: by the base change property
\[ \begin{split}
\Psi^n_1(t^{i_1},\ldots,t^{i_n})&=\Psi^n_1(1,\ldots,1)t^{i_1+\cdots+i_n}=(-1)^n\mu_n(t^{i_1},\ldots,t^{i_n}),\\
\Psi^n_{Id}(t^{i_1},\ldots,t^{i_n})&=\Psi^n_{Id}(1,\ldots,1)t^{i_1+\cdots+i_n}\,,
\end{split}\]
and then \eqref{equ.formulapsiuno} holds for $\K[t]$.
The passage from $\K[t]$ to any  $A$ is done by using the  standard polarization trick: given a finite sequence of homogeneous elements 
$a_1,\ldots,a_n\in A$ we consider the  algebra 
\[B=A[t_1,\ldots,t_n],\] 
where every $t_i$ is a central indeterminate of degree $|t_i|=-|a_i|$.
We have a morphism of unitary associative algebras 
\[ \alpha\colon \K[t]\to B,\qquad \alpha(t)=a_1t_1+\cdots+a_nt_n\,,\]
which by naturality gives  
\[ 
\Psi^{n}_{1}(\alpha(t),\ldots,\alpha(t))=\alpha\Psi^{n}_{1}(t,\ldots,t)=
\alpha((-1)^{n}t^n)
=(-1)^{n}\left(\sum_{i=1}^n a_it_i\right)^n,\]
while by  symmetry 
\[ \Psi^{n}_{1}(\alpha(t),\ldots,\alpha(t))=n!\Psi_1^n(a_1t_1,\ldots, a_nt_n)\,.\]
Looking at the coefficients of $t_1\cdots t_n$, in the first case we get  
\[ n!(-1)^{n}\mu_{n-1}(a_1t_1,\ldots, a_nt_n)
=n!(-1)^{n}(-1)^{\sum_{i<j}|e_i||a_j|}\,\mu_{n-1}(a_1,\ldots,a_n) t_1\cdots t_n,\]
whereas  in the second case, by base change property, we get  
\[ n! \,(-1)^{\sum_{i<j}|e_i||a_j|}\Psi^{n}_{1}(a_1,\ldots,a_n) t_1\cdots t_n,\]
and this concludes the proof of the first part of \eqref{equ.formulapsiuno}; the equality 
$\Psi_{Id}=\mu_{-1}$ is proved in the same way.
The map
\[ \Phi\colon\Aff^*(A)\to D(A),\qquad \Phi_u=\Psi_u-\Psi_{[u,1]}\;.\]
is a morphism of graded Lie algebras, since it is the composition of 
$\Psi$ with the Lie isomorphism 
\[\exp([\mu_{-1},-])\colon D_{-1}(A)\times D_0(A)\to D_{-1}(A)\times D_0(A)\;.\] 
In other words, for $a\in A$ and $f\colon A\to A$ we have 
\[ \Phi_a=\Psi_a,\qquad \Phi_f=\Psi_f-\Psi_{f(1)},\qquad 
\Psi_f=\Phi_f+\Phi_{f(1)}\,,\] 
and 
in particular
\begin{equation}\label{equ.phidellidentita} 
\Phi_{Id}=\Psi_{Id}-\Psi_1=\sum_{n\ge 0}(-1)^n\mu_n\,.
\end{equation} 

For every $a\in A$, the relation $[Id,a]=a$ gives
\[ \Phi_a=[\Phi_{Id},\Phi_a],\qquad \Phi_a^n=\sum_{h\ge 0}(-1)^h[\mu_h,\Phi_a^{n-h}]\] 
which, together the condition $\Phi^0_a=\Psi^0_a=a$ implies  
\begin{equation}
\Phi^0_a=a,\qquad \Phi_a^n=\frac{1}{n}\sum_{h=1}^n(-1)^{h+1}[\Phi_a^{n-h},\mu_h]\,.
\end{equation}

For every $f\colon A\to A$, we have 
$\Phi^0_f=\Psi^0_f-\Psi^0_{f(1)}=0$ and then, for every $a\in A$, the relation $\Phi_{[f,a]}=[\Phi_f,\Phi_a]$ gives 
\[ f(a)=\Phi^0_{[f,a]}=[\Phi^1_f,a]=\Phi^1_f(a)\]
proving that $\Phi^1_f=f$.
Moreover, the relation $[f,Id]=0$ gives 
\[\left[\Phi_f,\sum_{h\ge 0}(-1)^h\mu_h\right]=0\] 
and then the recursive formula
\begin{equation}
\Phi^0_f=0,\quad \Phi^1_f=f,\qquad \Phi^{n+1}_f=\frac{1}{n}\sum_{h=1}^n(-1)^{h+1}[\Phi^{n-h+1}_f,\mu_h]\;.
\end{equation} 
The proof of the unicity is complete.\end{proof}

For reference purposes it is convenient to collect as  a separate result the recursive formulas 
obtained in the proof of  Theorem~\ref{thm.unicity}.

\begin{theorem}\label{thm.recursive} 
The higher brackets $\Psi^n,\Phi^n$ are determined by the following recursive formulas: for every graded unitary associative algebra $A$,  every $x\in A$ and 
every $f\in \Hom^*_{\K}(A,A)$ we have 
\[ \Psi_x^n=\Phi_x^n,\qquad \Psi^n_f=\Phi^n_f+\Phi^n_{f(1)},\]
where 
\[\begin{split} \Phi^0_x=x,\qquad \Phi_x^n&=\frac{1}{n}\sum_{h=1}^n(-1)^{h+1}[\Phi_x^{n-h},\mu_h]\,,\\
\Phi^0_f=0,\qquad \Phi^1_f=f,\qquad \Phi^{n+1}_f&=\frac{1}{n}\sum_{h=1}^n(-1)^{h+1}[\Phi^{n-h+1}_f,\mu_h]\,.\end{split}\]
\end{theorem}

We shall prove in Proposition~\ref{prop.centralcase} that  
if $A$ is graded commutative then the operators  
$\Psi^n_f$ reduce to  the usual Koszul brackets as defined in \cite{koszul85}. Similarly the operators 
$\Phi^n_f$  are the higher brackets defined in \cite{akman,BDA} and called Koszul braces in \cite{markl1,markl2}.  
We shall refer to  the operators $\Psi^n$ as \emph{Koszul brackets} and to the operators
$\Phi^n$  as \emph{reduced Koszul brackets}.

\bigskip
\section{Examples and first properties of Koszul brackets}
\label{sec.firstproperties}

The brackets $\Phi^n,\Psi^n$ for low values of $n$  
can be easily computed by using the recursive formulas of Theorem~\ref{thm.recursive}.
For every $x\in A$ we have: 
\[ \begin{split}
\Psi^0_x=\Phi^0_x&=x,\\
\Psi^1_x=\Phi^1_x&=[x,\mu_1],\\
\Psi^2_x=\Phi^2_x&=\frac{1}{2}[[x,\mu_1],\mu_1]-\frac{1}{2}[x,\mu_2],\\
\Psi_x^3=\Phi^3_x&=\frac{1}{6}[[[x,\mu_1],\mu_1],\mu_1]-\frac{1}{6}[[x,\mu_2],\mu_1]
-\frac{1}{3}[[x,\mu_1],\mu_2]+\frac{1}{3}[x,\mu_3]\,.
\end{split}\]
For every $f\in \Hom^*_{\K}(A,A)$ we have: 
\[ \begin{split}
\Phi^1_f&=f,\\
\Phi^2_f&=[f,\mu_1],\\
\Phi^3_f&=\frac{1}{2}[[f,\mu_1],\mu_1]-\frac{1}{2}[f,\mu_2],\\
\Phi^4_f&=\frac{1}{6}[[[f,\mu_1],\mu_1],\mu_1]-\frac{1}{6}[[f,\mu_2],\mu_1]
-\frac{1}{3}[[f,\mu_1],\mu_2]+\frac{1}{3}[f,\mu_3]\,.\qquad
\end{split}\]
In the commutative case, the above formulas for $\Phi^2_f$ and $\Phi^3_f$ were already observed in \cite{FMpoisson}.
In a more explicit way, for  $a\in A$ we have: 
\[\begin{split} 
\Phi^1_x(a)&=\Phi^1_x(a)=-\frac{1}{2}(xa+(-1)^{|x||a|}ax),\\
\Phi^1_f(a)&=f(a),\qquad
\Psi^1_f(a)=f(a)-\frac{1}{2}(f(1)a+(-1)^{|f||a|}af(1))
\,.\end{split}\]
For  $x,a,b\in A$ we have:
\[ \Phi^2_x(a,b)=\frac{h(a,b)+(-1)^{|a||b|}h(b,a)}{2},\quad 
h(a,b)=\frac{xab+(-1)^{|a||x|}4axb+(-1)^{|x||ab|}abx}{6},\]
which can be written  in the  form:
\[ \Phi^2_x(a,b)=\frac{1}{12}(xab+(-1)^{|a||x|}4axb+(-1)^{|x||ab|}abx)+(-1)^{|a||b|}(a\rightleftarrows b).\]
In a similar way, for every $f\in \Hom^*_{\K}(A,A)$ and every $a,b\in A$ we get:
\[\begin{split} 
\Phi^2_f(a,b)=\;&\frac{f(ab)-f(a)b-(-1)^{|f|\,|a|}a f(b)}{2}+(-1)^{|a||b|}(a\rightleftarrows b)\,,\\
&\\
\Psi^2_f(a,b)=\;&\frac{f(ab)-f(a)b-(-1)^{|f|\,|a|}a f(b)}{2}+
\frac{f(1)ab+(-1)^{|a||f|}4af(1)b+(-1)^{|f||ab|}abf(1)}{12}\\
&+(-1)^{|a||b|}(a\rightleftarrows b)\vphantom{\frac{1}{1}}\,.\end{split}\]

\begin{lemma}\label{lem.reductionstep} 
For every $x,a_2,\ldots,a_n\in A$,  every $f\in \Hom^*_{\K}(A,A)$ and every $n>0$ we have:
\begin{enumerate}

\item $\Phi_x^n(1,a_2,\ldots,a_n)=-\Phi_x^{n-1}(a_2,\ldots,a_n)$,

\item $\Phi_f^n(1,a_2,\ldots,a_n)=\Phi_{f(1)}^{n-1}(a_2,\ldots,a_n)$,

\item $\Psi_f^n(1,a_2,\ldots,a_n)=0$.

\end{enumerate}
In particular $\Phi^n_f(1,\ldots,1)=(-1)^{n-1} f(1)$ and then $\Phi^n_f=0$ if and only if $\Psi^0_f=f(1)=0$ and $\Psi^n_f=0$.
\end{lemma}

\begin{proof} We have seen that $\Psi_{Id}=\Phi_{Id}+\Phi_{1}=\mu_{-1}$ and then, 
\[ \begin{split}
\Phi_x(1,a_2,\ldots,a_n)&=[\Phi_x,\mu_{-1}](a_2,\ldots,a_n)=\Phi_{[x,Id+1]}(a_2,\ldots,a_n)\\
&=-\Phi_x(a_2,\ldots,a_n)\,,\\
\Psi_f(1,a_2,\ldots,a_n)&=[\Phi_f,\mu_{-1}](a_2,\ldots,a_n)=\Psi_{[f,Id]}(a_2,\ldots,a_n)=0\,,\\
\Phi_f(1,a_2,\ldots,a_n)&=\Psi_f(1,a_2\ldots,a_n)-\Phi_{f(1)}(1,a_2,\ldots,a_n)\\
&=-\Phi_{f(1)}(1,a_2,\ldots,a_n)=\Phi_{f(1)}(a_2,\ldots,a_n)\,.
\end{split}\]
\end{proof}

\begin{example}[Derivations]\label{ex.derivations} 
Let $f\colon A\to A$ be a derivation,
then $[f,\mu_n]=\Psi^{n+1}_f=\Phi^{n+1}_f=0$ for every $n>0$.

In fact, assuming $f(ab)=f(a)b+(-1)^{|a||f|}af(b)$ for every $a,b\in A$,  a 
completely straightforward computation gives
$[f,\mu_1]=[f,\mu_2]=0$. According to  Lemma~\ref{lem.commutazionemuenne} and  Jacobi identity we have  then 
$[f,\mu_n]=0$ for every $n>2$.
The vanishing of $\Phi^n_f$ for $n\ge 2$ it is now an immediate consequence of 
Theorem~\ref{thm.recursive}, while the vanishing of $\Psi^n_f$ follows from the fact that $f(1)=0$.

The converse of the above implication  is generally false when $A$ is not graded commutative. 
Consider for instance 
the algebra $A=T(V)/I$, where $V$ is a vector space of dimension $\ge 2$, $T(V)=\bigoplus_{n\ge 0}V^{\otimes n}$ is the tensor algebra generated by $V$ and $I$ is the ideal generated by 
$V^{\otimes 3}$. Consider now a map $f\colon A\to A$ such that 
$f(1)=f(v)=f(u\otimes v+v\otimes u)=0$ for every $u,v\in V$. Since $f(1)=0$ we have 
$\Phi_f=\Psi_f$ and  it is  
easy to see that $[f,\mu_n]=\Psi^{n+1}_f=\Phi^{n+1}_f=0$ for every $n>0$.

\end{example}

\begin{example}[Left and right multiplication maps]\label{ex.leftright}
For a graded associative algebra $A$ and every $x\in A$ we shall denote by 
$L_x$ and $R_x$ the operators of left and right multiplication by $x$:
\[ L_x,R_x\colon A\to A,\qquad L_x(a)=xa,\quad R_x(a)=(-1)^{|a||x|}ax\;.\]
Denoting by $\{a,b\}=ab-(-1)^{|a||b|}ba$ the graded commutator in $A$, we have:
\[ L_x(1)=R_x(1)=x,\qquad \Psi^1_{L_x}(a)=\{x,a\},\quad \Psi^1_{R_x}(a)=\{a,x\},\]
\[ \Phi^2_{L_x}(a,b)=\Phi^2_{R_x}(a,b)=\frac{-1}{2}((-1)^{|a||x|}axb+(-1)^{(|a|+|x|)|b|}bxa)\,.\]
\[ \Psi^2_{L_x}(a,b)=\Psi^2_{R_x}(a,b)=\frac{1}{12}(\{\{x,a\},b\}+(-1)^{|a||b|}\{\{x,b\},a\})\,.\]
and then 
\[ \Psi^1_{L_x+R_x}=0,\qquad \Psi^2_{L_x+R_x}(a,b)=\frac{1}{6}(\{\{x,a\},b\}+(-1)^{|a||b|}\{\{x,b\},a\})\,.\]
Notice that $L_x-R_x=\{x,-\}$ is a derivation and then
\[ \Phi^n_{L_x}=\Phi^n_{R_x},\qquad \Psi^n_{L_x}=\Psi^n_{R_x},\]
for every $n\ge 2$.
\end{example}

\begin{example} As a partial converse of Example~\ref{ex.derivations} we have that if 
$A=T(V)$ is a tensor algebra and  
$f\colon A\to A$ is linear, then $f$ is a derivation if and only if $\Phi^2_f=0$. 

In fact, if $\Phi^2_f=0$, according to  the formula $\Phi^2_f(1,1)=-f(1)=0$ we have $f(1)=0$; 
replacing $f$ with $f-\delta$, where 
$\delta\colon A\to A$ is the (unique) derivation such that $\delta(v)=f(v)$ for every $v\in V$, it is not restrictive to assume  $f(V)=0$. 
For every $a\in V$, since $f(a)=0$, we have $0=\Phi^2_f(a,a)=f(a^2)$; by the same argument 
$0=\Phi^2_f(a^2,a)=f(a^3)$ and more generally $f(a^n)=0$ for every $n$.

Next we prove by induction on $n$ that $f(V^{\otimes n+1})=0$;
assuming $f(V^{\otimes i})=0$ for every $i\le n$, we need to prove $f(ab)=0$ for every
$a\in V$ and  $b=v_1\otimes\cdots\otimes v_n\in
V^{\otimes n}$. 
If $ab-ba=0$ then every $v_i$ is a scalar multiple of $a$, and therefore 
$f(ab)=cf(a^{n+1})=0$.
If $ab\not=ba$, then by the inductive assumption
\[0=2\Phi^2_f(a,b)=f(ab)+f(ba),\qquad  0=2\Phi^2_f(a^2,b)=f(a^2b)+f(ba^2)\;.\]
Moreover, the vanishing of $\Phi^2_f(ab,a)$ and $\Phi^2_f(a,ba)$ gives the equalities 
\[ f(aba)+f(a^2b)=f(ab)a+af(ab),\qquad f(aba)+f(ba^2)=f(ba)a+af(ba),\]
whose sum gives $f(aba)=0$ and therefore 
\[ f((ab)^2)+f((ba)^2)=f((aba)b)+f(b(aba))=0\,.\] 
\[ abf(ab)+f(ab)ab=f((ab)^2)=-f((ba)^2)=-baf(ba)-f(ba)ba=baf(ab)+f(ab)ba,\]
\[ (ab-ba)f(ab)+f(ab)(ab-ba)=0\,.\]
Since $ab-ba\not=0$ the last equality implies $f(ab)=0$.
\end{example}

\begin{proposition}\label{prop.centralcase}
Let $x\in A$ be a central element. 
Then 
\[\Psi_{xy}^n=L_x\Psi_y^n,\qquad \Psi_{L_xf}^n=L_x\Psi_f^n,\qquad
\Psi^{n+1}_f(x,a_1,\ldots,a_n)=\Psi^n_{[f,L_x]}(a_1,\ldots,a_n)\,,\]
for every $y\in A$, $f\colon A\to A$. For
 every sequence of central elements  
$x,c_1,\ldots,c_n\in A$ we have 
\[ \Psi^n_x(c_1,\ldots,c_n)=(-1)^n xc_1\cdots c_n,\qquad
\Psi^n_f(c_1,\ldots,c_n)=[...[[f,L_{c_1}],L_{c_2}]\ldots, L_{c_n}](1)\,,\]
and therefore, when $A$ is graded commutative the   
Koszul brackets $\Psi^n_f$ are the same of  the ones defined in \cite{koszul85}. 
\end{proposition}

\begin{proof} 
Since $x$ is central we have 
$[L_x\phi,\mu_n]=L_x[\phi,\mu_n]$ for every $n$ and every $\phi\in D(A)$; the first two formulas follow from 
this and Theorem~\ref{thm.recursive}.
In particular, for $f=Id=L_1$ we get $\Psi_{L_x}=L_x\Psi_{Id}$, viz. $\Psi^0_{L_x}=x\mu_{-1}$ and 
$\Psi^n_{L_x}=0$ for every $n>0$. This gives 
\[ \Psi_{[f,L_x]}=[\Psi_f,\Psi_{L_x}]=[\Psi_f,x\mu_{-1}]\;.\]
Since $[f,L_1]=0$ we get $\Psi^n_f(1,a_2,\ldots,a_n)=\Psi^{n-1}_{[f,L_1]}(a_2,\ldots,a_n)=0$ and, 
if $c_1,\ldots,c_n\in A$  are central elements, by induction on $n$ we have 
\[ \Psi^n_f(c_1,\ldots,c_n)=\Psi^0_{[...[[f,L_{c_1}],L_{c_2}]\ldots, L_{c_n}]}=
[...[[f,L_{c_1}],L_{c_2}]\ldots, L_{c_n}](1)\;.\]
\end{proof}

\begin{remark}
A  well known consequence of Proposition~\ref{prop.centralcase} is that, in the commutative case, the Koszul brackets are hereditary: this means that if $\Psi^n_f=0$ for some $n>0$, then $\Psi^k_f=0$ for every $k\ge n$.
This is clear if $n=1$ since $\Psi^1_f=0$ if and only if $f=L_{f(1)}$. 
If $\Psi_f^n=0$  for some $n>1$, then $\Psi^{n-1}_{[f,L_x]}=0$ for every $x$ and by induction 
$\Psi^{k-1}_{[f,L_x]}=0$ for every $x$ and every $k\ge n$.
According to Lemma~\ref{lem.reductionstep} also the reduced Koszul brackets $\Phi^n_f$ are hereditary. 
 
This is generally false in the non-commutative case. Consider for instance an element $x\in A$ of degree $0$ and the operator $f=L_x+R_x\colon A\to A$, $f(a)=ax+xa$, for which we have already seen in Example~\ref{ex.leftright} that
\[ \Psi^1_f=0,\qquad \Psi^2_f(a,b)=\frac{1}{6}(\{\{x,a\},b\}+(-1)^{|a||b|}\{\{x,b\},a\})\,.\]
\end{remark}

\begin{remark} If $A$ is  not commutative and $d\colon A\to A$ is a derivation, then  in general 
$\Phi^3_{d^2}\not=0$. Therefore 
the attempt to use Koszul higher brackets to 
define differential operators remains  unsatisfactory, 
although slightly better than the trivial extension of  
Grothendieck's definition, cf. \cite[Rem. 2.3.5]{GinzSche}.
\end{remark}

\bigskip
\section{The formulas of Bering and Bandiera}
\label{sec.beringbandiera}

For a given linear endomorphism $f\colon A\to A$, 
the brackets $\Psi^n_f$ were defined by Koszul \cite{koszul85} in the commutative case by the formula
\[\begin{split}
\Psi^n_f(a_1,\ldots,a_n)&=[...[[f,L_{a_1}],L_{a_2}]\ldots, L_{a_n}](1)\\
&=\sum_{k=0}^n\frac{(-1)^{n-k}}{k!(n-k)!}\sum_{\pi\in \Sigma_n}\epsilon(\pi)\; 
f(1\cdot a_{\pi(1)}\cdots a_{\pi(k)})a_{\pi(k+1)}\cdots a_{\pi(n)}\,,
\end{split}\] 
where the $1$ inside the argument of $f$ is the unit of $A$. The
generalized Jacobi identities  
\begin{equation}\label{equ.basicidentity}
\Psi^{n}_{[f,g]}=\sum_{i=0}^{n+1}[\Psi_f^i,\Psi_g^{n-i+1}]
\end{equation} 
were first proved  independently in \cite{BDA,kravchenko2000}. Extensions of the brackets $\Psi^n_f$ to the 
non-commutative case satisfying \eqref{equ.basicidentity} were first 
proposed by Bering \cite{Bering} and later, with a completely different approach, by Bandiera \cite{derived} as a consequence of a more general formula about derived brackets. 
Both approaches  involve the sequence
$B_n$ of  Bernoulli numbers:
\[ \begin{split}
B(x)&=\sum_{n\ge 0}B_n\frac{x^n}{n!}=\frac{x}{e^x-1}\\
&=1-\frac{1}{2}\,\frac{x}{1!}+\frac{1}{6}\,\frac{x^2}{2!}-\frac{1}{30}\,\frac{x^4}{4!}
+\frac{1}{42}\,\frac{x^6}{6!}-\frac{1}{30}\,\frac{x^8}{8!}+\frac{5}{66}\,\frac{x^{10}}{10!}+\cdots\,.\end{split}\]
In order to  simplify the notation it is useful to introduce the rational numbers
\[ B_{i,j}=(-1)^j\sum_{k=0}^j \binom{j}{k}B_{i+k}\,,\qquad i,j\ge 0\,,\]
together their exponential generating function 
\[ B(x,y)=\sum_{i,j\ge 0}B_{i,j}\frac{x^iy^j}{i!\,j!}\in \Q[[x,y]]\,.\]
Since  
\[\begin{split}
B(x,y)=&\sum_{i,j}(-1)^j\sum_{k=0}^j B_{i+k}\binom{j}{k}\frac{x^iy^j}{i!j!}
=\sum_{i,j}\sum_{k=0}^j \frac{(-y)^{j-k}}{(j-k)!}B_{i+k}\frac{x^i(-y)^k}{i!\,k!}\\
=&\;e^{-y}B(x-y)
=\frac{x-y}{e^{x}-e^{y}}\\
=&\; 1-\left(\frac{x}{2}+\frac{y}{2}\right)
+\frac{1}{2!}\left(\frac{x^{2}}{6}+\frac{2xy}{3}+\frac{y^{2}}{6}\right)
-\frac{1}{3!}\left(\frac{x^{2}y}{2}+\frac{xy^{2}}{2}\right) \\
&+\frac{1}{4!}\left(-\frac{x^{4}}{30}+\frac{2x^{3}y}{15}
+\frac{4x^{2}y^{2}}{5}+\frac{2xy^{3}}{15}-\frac{y^{4}}{30}\right)+\cdots\,,
\end{split}\]
we have $B(x,y)=B(y,x)$ and therefore  $B_{i,j}=B_{j,i}$ for every $i,j$; as a byproduct we 
have just proved the following (well known) formulas about Bernoulli numbers:
\begin{equation}\label{equ.3bis}
(-1)^j\sum_{k=0}^j \binom{j}{k}B_{i+k}=(-1)^i\sum_{k=0}^i \binom{i}{k}B_{j+k},\qquad i,j\ge 0\,.
\end{equation}

\begin{equation}\label{equ.3ter}
\sum_{k=0}^n \binom{n}{k}B_{k}=(-1)^n\sum_{k=0}^0 \binom{0}{k}B_{n+k}=(-1)^nB_n,\qquad n\ge 0\,.
\end{equation}

For $x=y$ we get 
\[ \sum_{i,j=0}^{\infty}B_{i,j}\frac{x^{i}}{i!}\frac{x^{j}}{j!}=
e^{-x}\sum_{n=0}^{\infty}\frac{B_n}{n!}(x-x)^n=e^{-x}\]
and then 
\[ \sum_{i=0}^{n}B_{i,n-i}\frac{x^{n}}{i!}\frac{x^{j}}{j!}=(-1)^n\frac{x^n}{n!},\qquad 
\sum_{i=0}^{n}\binom{n}{i}B_{i,n-i}=(-1)^n\,.\]

\begin{lemma}\label{lem.sommabigei} 
The numbers $B_{i,j}$ are uniquely determined by the following properties:
\begin{enumerate}
\item $B_{0,n}=B_n$;

\item $B_{i,j}+B_{i+1,j}+B_{i,j+1}=0$ for every $i,j\ge 0$.
\end{enumerate}
\end{lemma}

\begin{proof} The only nontrivial part is  the proof that 
$B_{i,j}+B_{i+1,j}+B_{i,j+1}=0$. This can be done either by applying binomial identities
to the formulas  
$B_{i,j}=(-1)^j\sum_{k=0}^j \binom{j}{k}B_{i+k}$, or by applying the differential operator $\dfrac{\partial~}{\partial x}+\dfrac{\partial~}{\partial y}$ to the equality
$e^yB(x,y)=B(x-y)$. 
\end{proof}

We are now ready to prove that both Bering and Bandiera's brackets coincide with the brackets defined in Section~\ref{sec.uniquenesstheorem}; the key technical points of the proof will be the two lemmas 
 \ref{lem.reductionstep} and \ref{lem.sommabigei}.

\begin{theorem}[Bering's formulas]\label{thm.beringformula} 
In the above notation, for every $x\in A$ and every $f\in \Hom^*_{\K}(A,A)$ we have 
\[\begin{split}
\Psi_{x}^{n}(a_{1},\ldots, a_{n})
&=\!\!\sum_{i,j\ge 0}^{i+j=n}\frac{B_{i,j}}{\;i!\,j!\;}
\sum_{\pi\in \Sigma_{n}}
\epsilon(\pi,i,x)\,
a_{\pi(1)}\cdots a_{\pi(i)}\,x\, a_{\pi(i+1)}\cdots a_{\pi(n)}\,,\\
\Psi_{f}^{n}(a_{1},\ldots, a_{n})
&=\!\!\!\!\sum_{i,j,k\ge 0}^{i+j+k=n}\frac{B_{i,j}}{i!\,j!\,k!}
\sum_{\pi\in \Sigma_{n}}
\epsilon(\pi,i,f)\,
a_{\pi(1)}\cdots a_{\pi(i)}\cdot\\
&\qquad\qquad\qquad\qquad\quad\;\cdot\! f(1\!\cdot\! a_{\pi(i+1)}\cdots a_{\pi(i+k)})\, a_{\pi(i+k+1)}\cdots a_{\pi(n)}\,,
\end{split}\]
where 
\[\epsilon(\pi,i,x)=\epsilon(\pi)\!\cdot\! (-1)^{|x|(|a_{\pi(1)}|+\cdots+|a_{\pi(i)}|)},\qquad
\epsilon(\pi,i,f)=\epsilon(\pi)\!\cdot\! (-1)^{|f|(|a_{\pi(1)}|+\cdots+|a_{\pi(i)}|)},\] are the usual Koszul signs.
\end{theorem}

\begin{proof}
A direct inspection shows that the above formulas are true for $n=0,1,2$. In general, expanding the recursive equations of Theorem~\ref{thm.recursive} we get  
\[\Phi_{x}^{n}(a_{1},\ldots, a_{n})
=\!\!\sum_{i,j\ge 0}^{i+j=n}\frac{C_{i,j}}{i!\,j!}
\sum_{\pi\in \Sigma_{n}}
\epsilon(\pi,i,x)
a_{\pi(1)}\cdots a_{\pi(i)}\,x\, a_{\pi(i+1)}\cdots a_{\pi(n)}\,,
\]
for a suitable sequence of rational numbers $C_{i,j}$. In order to prove the theorem it is sufficient to prove that $C_{0,n}=B_n$ for every $n\ge 2$ and $C_{i,j}+C_{i+1,j}+C_{i,j+1}=0$ for every $i,j\ge 0$; to this end it is not restrictive to assume that $A$ is the free unitary 
associative algebra generated by $x,a_1,\ldots,a_n$, $|x|=|a_i|=0$.
The coefficient of $a_1\cdots a_i\, x\, a_{i+1}\cdots a_n$ in 
$\Phi_{x}^{n+1}(1,a_{1},\ldots, a_{n})$ is equal to 
\[ (i+1)\frac{C_{i+1,j}}{(i+1)!j!}+(j+1)\frac{C_{i,j+1}}{i!(j+1)!}=\frac{C_{i+1,j}+C_{i,j+1}}{i!\,j!}\;.\]
According to Lemma~\ref{lem.reductionstep},   
$\Phi_{x}^{n+1}(1,a_{1},\ldots, a_{n})=-\Phi_{x}^{n}(a_{1},\ldots, a_{n})$ and then the above 
coefficient is equal to $-\frac{C_{i,j}}{i!j}$.
 
The proof that $C_{0,n}=B_n$ is done by induction on $n$. Assuming $C_{0,i}=B_i$ for every $i<n$, 
the coefficient of  $x\,a_1\cdots a_n$ in 
$[\Phi^{n-h}_x,\mu_h]$, $h>0$,  is equal to 
\[ \frac{B_{n-h}}{(n-h)!}\frac{n-h-1}{(h+1)!},\]
and therefore 
\[ \frac{C_{0,n}}{n!}=\frac{1}{n}\sum_{h=1}^n(-1)^{h+1}\frac{B_{n-h}}{(n-h)!}\frac{n-h-1}{(h+1)!}\,.\]
Since $(-1)^{h+1}B_{n-h}(n-h-1)=(-1)^{n+1}B_{n-h}(n-h-1)$ for every $h$ we have
\[\begin{split} 
(-1)^{n+1}C_{0,n}&=(n-1)!\sum_{h=1}^n\frac{B_{n-h}}{(n-h)!}\frac{n-h-1}{(h+1)!}\\
&=n!\sum_{h=1}^n\frac{B_{n-h}}{(n-h)!(h+1)!}-(n-1)! \sum_{h=1}^n\frac{B_{n-h}}{(n-h)!h!}\\
&=\frac{1}{n+1}\sum_{h=1}^n\binom{n+1}{n-h}B_{n-h}-\frac{1}{n}\sum_{h=1}^n\binom{n}{n-h}B_{n-h}\\
&=\frac{1}{n+1}\sum_{s=0}^{n-1}\binom{n+1}{s}B_{s}-\frac{1}{n}\sum_{s=0}^{n-1}\binom{n}{s}B_s
\,.\end{split}\]
Since $n\ge 2$ we have 
\[ \sum_{s=0}^{n}\binom{n+1}{s}B_{s}=\sum_{s=0}^{n-1}\binom{n}{s}B_s=0\]
and therefore 
\[ (-1)^{n+1}C_{0,n}=-\frac{1}{n+1}\binom{n+1}{n}B_n=-B_n,\qquad C_{0,n}=(-1)^nB_n=B_n\;.\]
Again by  Theorem~\ref{thm.recursive} we have 
\[\begin{split}
\Phi_{f}^{n}(a_{1},\ldots, a_{n})
=\sum_{k=1}^n\sum_{i,j\ge 0}^{i+j=n-k}\frac{C_{i,j,k}}{i!\,j!\,k!}
\sum_{\pi\in \Sigma_{n}}
&{\epsilon(\pi,i,f)}
a_{\pi(1)}\cdots a_{\pi(i)}\cdot\\
&\;\cdot f(1\!\cdot\! a_{\pi(i+1)}\cdots a_{\pi(i+k)})\cdot a_{\pi(i+k+1)}\cdots a_{\pi(n)}\,,
\end{split}\]
for a suitable sequence of rational numbers $C_{i,j,k}$. Therefore
\[\begin{split}
\Psi_{f}^{n}(a_{1},\ldots, a_{n})
=\sum_{k=0}^n\sum_{i,j\ge 0}^{i+j=n-k}\frac{C_{i,j,k}}{i!\,j!\,k!}
\sum_{\pi\in \Sigma_{n}}
&{\epsilon(\pi,i,f)}
a_{\pi(1)}\cdots a_{\pi(i)}\cdot\\
&\;\cdot f(1\cdot a_{\pi(i+1)}\cdots a_{\pi(i+k)})\cdot a_{\pi(i+k+1)}\cdots a_{\pi(n)}\,,
\end{split}\]
where $C_{i,j,0}=B_{i,j}$ for every $i,j\ge 0$; in order to prove the theorem it is therefore sufficient to show by induction on $k$ that  $C_{i,j,k}=B_{i,j}$. By Lemma~\ref{lem.reductionstep} we have  
$\Psi^{n+1}_f(1,a_1,\ldots,a_n)=0$, whereas the coefficient of  
$a_{1}\cdots a_{i}\cdot f(1\cdot a_{i+1}\cdots a_{i+k})\cdot a_{i+k+1}\cdots a_{n}$ in 
$\Psi^{n+1}_f(1,a_1,\ldots,a_n)$ is equal to 
\[ 0=(i+1)\frac{C_{i+1,j,k}}{(i+1)!\,j!\,k!}+(j+1)\frac{C_{i,j+1,k}}{i!(j+1)!\,k!}+
(k+1)\frac{C_{i,j,k+1}}{i!\,j!\,(k+1)!}\;,\]
and then
\[ C_{i,j,k+1}=-C_{i+1,j,k}-C_{i,j+1,k}=-B_{i+1,j}-B_{i,j+1}=B_{i,j}\,.\]
\end{proof}

As already pointed out in \cite{Bering,voronov}, although  
the expression
$[...[[f,L_{a_1}],L_{a_2}]\ldots, L_{a_n}](1)$ 
makes sense in every unitary graded associative algebra,  
its symmetrization does not give the Koszul brackets $\Psi^n_f$. 
For instance, if $|a|=|b|=|f|=0$ and $f(1)=0$, then 
\[ \frac{[[f,L_a],L_b]+(-1)^{|a||b|}[[f,L_b],L_a]}{2}-\Psi^2_f(a,b)=
\frac{\{f(a),b\}+\{f(b),a\}}{2}\,,\]
where $\{-,-\}$ denotes the graded commutator in $A$.

\begin{theorem}[Bandiera's formulas] Let $A$ be a graded unitary associative algebra. 
For every $x\in A$ and every $f\in \Hom^*_{\K}(A,A)$ we have:
\[\begin{split} 
\Psi_x^n(a_1,\ldots,a_n)&=\!\sum_{\pi\in \Sigma_n}\varepsilon(\pi)\sum_{k=0}^{n
}\frac{(-1)^nB_{n-k}}{k!(n-k)!}\{\{\cdots\{x\,a_{\pi(1)}\cdots a_{\pi(k)},
a_{\pi(k+1)}\},\ldots\},a_{\pi(n)}\}\;,\\
\Psi_f^n(a_1,\ldots,a_n)&=\!\sum_{\pi\in \Sigma_n}\varepsilon(\pi)\sum_{k=0}^{n
}\frac{B_{n-k}}{k!(n-k)!}\{\{\cdots\{f_k(a_{\pi(1)},\ldots,a_{\pi(k)}),
a_{\pi(k+1)}\},\ldots\},a_{\pi(n)}\}\;,\end{split} \]
where $\{a,b\}=ab-(-1)^{|a||b|}ba$, and $f_k\colon A^{\otimes k}\to A$ is the sequence of  operators
\[ f_0=f(1),\qquad f_n(a_1,\ldots,a_n)=[\cdots[[f,L_{a_1}],L_{a_2}],\ldots, L_{a_n}](1)\,.\]
\end{theorem}

\begin{proof}
Denoting momentarily by 
\[\begin{split} 
\Theta_x^n(a_1,\ldots,a_n)&=\!\sum_{\pi\in \Sigma_n}\varepsilon(\pi)\sum_{k=0}^{n
}\frac{(-1)^nB_{n-k}}{k!(n-k)!}\{\{\cdots\{x\,a_{\pi(1)}\cdots a_{\pi(k)},
a_{\pi(k+1)}\},\ldots\},a_{\pi(n)}\}\;,\\
\Theta_f^n(a_1,\ldots,a_n)&=\!\sum_{\pi\in \Sigma_n}\varepsilon(\pi)\sum_{k=0}^{n
}\frac{B_{n-k}}{k!(n-k)!}\{\{\cdots\{f_k(a_{\pi(1)},\ldots,a_{\pi(k)}),
a_{\pi(k+1)}\},\ldots\},a_{\pi(n)}\}\;,\end{split} \]
by Bering's formulas it is sufficient  to prove that 
\begin{equation}\label{equ.bandier1}
\Theta_{x}^{n}(a_{1},\ldots, a_{n})
=\!\!\sum_{i,j\ge 0}^{i+j=n}\frac{B_{i,j}}{i!\,j!}
\sum_{\pi\in \Sigma_{n}}
\epsilon(\pi,i,x)\,
a_{\pi(1)}\cdots a_{\pi(i)}\,x\, a_{\pi(i+1)}\cdots a_{\pi(n)}\,,
\end{equation}
\begin{equation}\label{equ.bandier2}
\begin{split}
\Theta_{f}^{n}(a_{1},\ldots, a_{n})
&=\!\!\sum_{i,j,k\ge 0}^{i+j+k=n}\frac{B_{i,j}}{i!\,j!\,k!}
\sum_{\pi\in \Sigma_{n}}
\epsilon(\pi,i,f)\,
a_{\pi(1)}\cdots a_{\pi(i)}\cdot\\
&\qquad\qquad\qquad\qquad\cdot f(1\cdot a_{\pi(i+1)}\cdots a_{\pi(i+k)})\, a_{\pi(i+k+1)}\cdots a_{\pi(n)}\,.
\end{split}\end{equation}
It's easy to see that, for every $a_1,\ldots,a_n\in A$, we have:
\[  \Theta_f^{n+1}(1,a_2,\ldots,a_n)=0,\qquad \Theta_x^{n+1}(1,a_1,\ldots,a_n)=-\Theta_x^{n}(a_1,\ldots,a_n)\,,\]
and then we can prove \eqref{equ.bandier1} and \eqref{equ.bandier2} in the same way as in Theorem~\ref{thm.beringformula}. In fact, expanding  the commutator brackets we can write
\[\Theta_{x}^{n}(a_{1},\ldots, a_{n})
=\!\!\sum_{i,j\ge 0}^{i+j=n}\frac{C_{i,j}}{i!\,j!}
\sum_{\pi\in \Sigma_{n}}
\epsilon(\pi,i,x)
a_{\pi(1)}\cdots a_{\pi(i)}\,x\, a_{\pi(i+1)}\cdots a_{\pi(n)}\,,
\]
for some rational coefficients $C_{i,j}$. Looking at 
the coefficient of $xa_1\cdots a_n$ in 
$\Theta_x^{n}(a_1,\ldots,a_n)$ we get 
\[ \frac{C_{0,n}}{n!}=(-1)^n\sum_{k=0}^n\frac{B_{n-k}}{k!(n-k)!}=\frac{(-1)^n}{n!}
\sum_{k=0}^n\binom{n}{k}B_{k}=\frac{B_n}{n!}=\frac{B_{0,n}}{n!}\;,\]
and the proof of $C_{i,j}=B_{i,j}$ follows from the equality 
$\Theta_x^{n+1}(1,a_1,\ldots,a_n)=-\Theta_x^{n}(a_1,\ldots,a_n)$ exactly as in 
Theorem~\ref{thm.beringformula}. Similarly,
expanding the operators $f_k$ and the commutator brackets we can write
\[\begin{split}
\Theta_{f}^{n}(a_{1},\ldots, a_{n})
=\sum_{k=0}^n\sum_{i,j\ge 0}^{i+j=n-k}\frac{C_{i,j,k}}{i!\,j!\,k!}
\sum_{\pi\in \Sigma_{n}}
&{\epsilon(\pi,i,f)}
a_{\pi(1)}\cdots a_{\pi(i)}\cdot\\
&\;\cdot f(1\cdot a_{\pi(i+1)}\cdots a_{\pi(i+k)})\cdot a_{\pi(i+k+1)}\cdots a_{\pi(n)}\,,
\end{split}\]
for certain rational coefficients $C_{i,j,k}$. Comparing the coefficients of 
$f(1\cdot a_{1}\cdots a_{k})\cdot a_{k+1}\cdots a_{n}$ we get
\[ \frac{C_{0,n-k,k}}{k!\,(n-k)!}=\frac{B_{n-k}}{k!\,(n-k)!},\qquad C_{0,n-k,k}=B_{n-k}=B_{0,n-k}\;.\]
We now prove by induction on $i$ that $C_{i,j,k}=B_{i,j}$ and therefore the proof of 
the equality $\Psi_f=\Theta_f$  follows by Bering's formula; the coefficient of  
$a_{1}\cdots a_{i}\cdot f(1\cdot a_{i+1}\cdots a_{i+k})\cdot a_{i+k+1}\cdots a_{n}$ in 
$\Theta^{n+1}_f(1,a_1,\ldots,a_n)$ is equal to 
\[ 0=(i+1)\frac{C_{i+1,j,k}}{(i+1)!\,j!\,k!}+(j+1)\frac{C_{i,j+1,k}}{i!(j+1)!\,k!}+
(k+1)\frac{C_{i,j,k+1}}{i!\,j!\,(k+1)!}\;,\]
and then
\[ C_{i+1,j,k}=-C_{i,j,k+1}-C_{i,j+1,k}=-B_{i,j}-B_{i,j+1}=B_{i+1,j}\,.\]
\end{proof}

\bigskip
\section{Additional  remarks}

\subsection*{The uniqueness theorem for non-unitary algebras}

It is clear from the proof that Theorem~\ref{thm.unicity} admits several slight modifications, either changing the underlying categories or the choice of initial terms and gauge fixing conditions. 
According to Theorem~\ref{thm.recursive}, the reduced Koszul brackets $\Phi^n$ also make sense for graded associative algebras without unit. It is therefore natural to expect an uniqueness theorem also for reduced Koszul brackets in the setup of non-unitary associative algebras.

\begin{theorem}\label{thm.unicity2}
There exists a unique way to assign to every graded associative algebra $A$ a morphism of graded Lie algebras
\[ \Phi\colon\End^*(A)\to D(A),\qquad \Phi=\sum \Phi^n,\quad \Phi^n\colon\End^*(A)\to D_{n-1}(A),\] 
such that the following conditions are satisfied:

\begin{enumerate}

\item \emph{naturality:} for every morphism of graded associative algebras  
$\alpha\colon A\to B$ and every pair of linear maps $f\colon A\to A$, $g\colon B\to B$ such that 
$g\alpha=\alpha f$, we have
\[ 
\alpha\Phi^n_f=\Phi^n_g\alpha^{\odot n}\quad\colon A^{\odot n}\to B\,.\]

\item \emph{base change}:  the operators $\Phi^n_{Id}$ are multilinear over the graded centre of $A$: 
more precisely, if $c\in A$ and $ac=(-1)^{|a||c|}ca$ for every $a\in A$, then 
\[ \Phi^n_{Id}(a_1,\ldots,a_nc)=\Phi^n_{Id}(a_1,\ldots,a_n)c\]
for every $a_1,\ldots,a_n$.

\item \emph{initial terms:} for every $f\colon A\to A$ we have 
\[\Phi^0_f=0,\qquad \Phi^1_f=f\;.\]

\item \emph{gauge fixing:} at least one of the following conditions is satisfied:
\begin{enumerate}

\item if $A=\K$ then $\Phi_{Id}=\sum_{n\ge 0}(-1)^n\mu_n$;

\item if $A=\K[t]$  and $\de^2$ is the second derivative operator, then 
$\Phi^2_{\de^2}(a,b)=2(\de a)(\de b)$ and  $\Phi^n_{\de^2}=0$ for every $n\ge 3$.

\end{enumerate}

\end{enumerate}
\end{theorem}

\begin{proof} 
Notice first that when the graded associative algebra $A$ is not unitary, then the definition of $\mu_n\colon A^{\odot n+1}\to A$ makes sense only for $n\ge 0$ and Lemma~\ref{lem.commutazionemuenne} holds for every $n,m\ge 0$.

The existence is clear: it is sufficient to take
\[ \Phi_f=\exp\left(\left[-,\sum_{n=1}^{\infty}K_n \mu_n
\right]\right)f\]
where $K_n$ is the sequence of rational numbers defined in the proof of Theorem~\ref{thm.unicity}: 
when $A$ is unitary we recover the reduced Koszul brackets. In particular,  
for $A=\K[t]$ we have $\de^2(1)=0$,  
\[ \Phi^n_{\de^2}(a_1,\ldots,a_n)=\Psi^n_{\de^2}(a_1,\ldots,a_n)=[...[[\de^2,L_{a_1}],L_{a_2}]\ldots, L_{a_n}](1)\]
and then $\Phi^n_{\de^2}=0$ for every $n\ge 3$. The recursive formulas of Theorem~\ref{thm.recursive} give
$[\Phi^2_{\de^2},\mu_{n-1}]=[\de^2,\mu_n]$ for every $n$.

The proof of the unicity 
is essentially the same as in the the unitary case and we give only a sketch.
Assume that for every graded associative algebra $A$ it is given a morphism of graded Lie algebras 
$\phi\colon \End^*(A)\to D(A)$ which satisfies the condition of the theorem: we want to prove that $\phi=\Phi$. 

For $A=\K$ we have $\phi_{Id}=\mu_0+\sum_{n\ge 1}s_n\mu_n$ for a suitable sequence $s_1,s_2,\ldots\in \K$ 
and  by base change the same holds for $A=\K[t]$. The gauge fixing condition implies that 
$s_n=(-1)^n$ for every $n$: this is clear in the first case, while in the second case 
we have $[\phi^2_{\de^2},\mu_{n-1}]=[\Phi^2_{\de^2},\mu_{n-1}]=[\de^2,\mu_n]$ for every $n$ and from the equality 
\[ 0=[\phi_{\de^2},\phi_{Id}]=
[\phi^2_{\de^2},s_{n-1}\mu_{n-1}]+[\de^2,s_n\mu_n]\] 
we get $s_n=-s_{n-1}$ for every $n$.

Considering the inclusion $t\K[t]\to \K[t]$, by naturality the formula $\phi_{Id}=\sum_{n\ge 0}(-1)^n\mu_n$ holds also for the algebra $t\K[t]$ and the polarization trick gives $\phi_{Id}=\sum_{n\ge 0}(-1)^n\mu_n$ for every graded associative algebra $A$.
Finally the formula 
$[\phi_f,\phi_{Id}]=\phi_{[f,Id]}=0$ gives immediately the recursive equations
\[\phi^0_f=0,\quad \phi^1_f=f,\qquad \phi^{n+1}_f=\frac{1}{n}\sum_{h=1}^n(-1)^{h+1}[\phi^{n-h+1}_f,\mu_h]\;.\]
\end{proof}

\bigskip

\subsection*{The quantum antibracket of Vinogradov, Batalin and Marnelius.}
Given a graded associative algebra $A$, the quantum antibracket associated to a homogeneous element $Q\in A$ of odd degree $|Q|=k$ is defined by the formula
\[ (-,-)_Q\colon A\times A\to A,\qquad (a,b)_Q=\frac{1}{2}(\{a,\{Q,b\}\}-(-1)^{(|a|+k)(|b|+k)}
\{b,\{Q,a\}\})\;.\]
This bracket has been introduced by Batalin and Marnelius in \cite{BM} as the unique bracket (up to a scalar factor) satisfying certain natural properties arising in the context of quantization of classical dynamic. An essentially equivalent  construction was also given by Vinogradov in the algebra of linear endomorphisms of the space of differential forms on a manifold, with $Q=d$ the de Rham differential 
\cite{vino,KS04}.

The bracket $(-,-)_Q$ is graded skewsymmetric of degree $0$ on the shifted complex $A[-k]$ and corresponds, by standard shifting degrees (d\'ecalage) formulas, to the graded symmetric operator $B_Q\colon A^{\odot 2}\to A$
of degree $k$:
\[
B_Q(a,b)=(-1)^{k|a|}(a,b)_Q=-\frac{1}{2}\left(\{\{Q,a\},b\}+
(-1)^{|a||b|}
\{\{Q,b\},a\}\right)
\,.\]
Therefore, according to Example~\ref{ex.leftright}, we have 
\[ \Psi^2_{L_Q}=\Psi^2_{R_Q}=-\frac{1}{6}B_Q\,.\]

\subsection*{Gauge fixing variation and B\"{o}rjeson's brackets.}

Changing the gauge fixing condition in Theorem~\ref{thm.unicity2} one can obtain different hierarchies of higher brackets: for instance, setting $\Phi_{Id}=\mu_0$ we get $\Phi^n=0$ for every $n\ge 2$, while   
setting  $\Phi_{Id}=\mu_0-\mu_1$, i.e., $K_1=1$ and  $K_n=0$ for every $n>1$,  
it is easy to see that the resulting higher brackets are the graded symmetrizations of the 
B\"{o}rjeson's brackets \cite{borjeson,markl1}.

\end{document}